\documentclass[12pt, reqno]{amsart}
\usepackage{amsmath, amsthm, amscd, amsfonts, amssymb, graphicx, color}
\usepackage[bookmarksnumbered, colorlinks, plainpages]{hyperref}
\hypersetup{colorlinks=true,linkcolor=red, anchorcolor=green, citecolor=cyan,
urlcolor=red, filecolor=magenta, pdftoolbar=true}

\newtheorem{theorem}{Theorem}[section]
\newtheorem{lemma}[theorem]{Lemma}

\newtheorem{corollary}[theorem]{Corollary}
\theoremstyle{definition}

\theoremstyle{remark}

\numberwithin{equation}{section}

\begin{document}

\setcounter{page}{1}

\title[Norm--parallelism in spaces of continuous functions]
{Characterizations of norm--parallelism in spaces of continuous functions}

\author[A. Zamani]{Ali Zamani}
\address{Department of Mathematics, Farhangian University, Tehran, Iran.}
\email{zamani.ali85@yahoo.com}


\begin{abstract}
In this paper, we consider the characterization of norm--parallelism
problem in some classical Banach spaces. In particular, for two
continuous functions $f, g$ on a compact Hausdorff space $K$,
we show that $f$ is norm--parallel to $g$ if and only if there exists a probability
measure (i.e. positive and of full measure equal to $1$) $\mu$ with its support contained
in the norm attaining set $\{x\in K: \, |f(x)| = \|f\|\}$ such that
$\big|\int_K \overline{f(x)}g(x)d\mu(x)\big| = \|f\|\,\|g\|$.
\end{abstract}
\maketitle
\section{Introduction and preliminaries}
Let $(X, \|\cdot\|)$ be a normed space and denote, as usual, by
$\mathbb{B}_X$ and $\mathbb{S}_X$ its closed unit ball and unit sphere, respectively,
and denote the topological dual of $X$ by $X^*$.
Let $\mathcal{C}_b(\Omega)$ and $\mathcal{C}(K)$ denote the
Banach spaces of all bounded continuous functions on a locally
compact Hausdorff space $\Omega$, with the usual norm
$\|f\| = \displaystyle{\sup_{x\in \Omega}}|f(x)|\, (f\in \mathcal{C}_b(\Omega))$
and all continuous functions on a compact Hausdorff space $K$,
with the usual norm $\|f\| = \displaystyle{\max_{x\in K}}|f(x)|\, (f\in \mathcal{C}(K))$,
respectively. By $\mathcal{C}_u(\mathbb{B}_X,X)$ we denote the space of all uniformly
continuous $X$ valued functions on $\mathbb{B}_X$ endowed with the supremum norm.
Given a bounded function $f : \,\mathbb{S}_X \longrightarrow X$,
its numerical radius is
$$v(f) : = \sup\{|x^*(f(x))|\,: \,\,\|x^*\| = x^*(x) = 1\}.$$
Let us comment that for a bounded function $f : \,\mathbb{B}_X \longrightarrow X$, the above
definitions apply by just considering $v(f) := v(f|_{\mathbb{S}_X}$).

Recall that an element $x\in X$ is said to be norm--parallel to another element
$y\in X$ (see \cite{S, Z.M.1}), in short
$x\parallel y$, if
\begin{align*}
\|x+\lambda y\|=\|x\|+\|y\| \qquad \mbox{for some}\,\,
\lambda\in\mathbb{T}.
\end{align*}
Here, as usual, $\mathbb{T}$ is the unit circle of the complex plane.
In the framework of inner product spaces, the norm--parallel relation
is exactly the usual vectorial parallel relation, that is,
$x\parallel y$ if and only if $x$ and $y$ are linearly dependent.
In the setting of normed linear spaces, two linearly
dependent vectors are norm--parallel, but the converse is false in general.
Notice that the norm--parallelism is symmetric and $\mathbb{R}$-homogenous,
but not transitive (i.e., $x\parallel y$ and $y\parallel z \nRightarrow x\parallel z$;
see \cite[Example 2.7]{Z.M.2}, unless $X$ is smooth at $y$; see \cite[Theorem 3.1]{W}).

In the context of continuous functions,
the well-known Daugavet equation
\begin{align*}
\|T + Id\| = \|T\| + 1
\end{align*}
is a particular case of parallelism.
Here $Id$ denotes, as usual, the identity function.
Applications of this equation arise in solving a variety of problems
in approximation theory; see \cite{We} and the references therein.

Some characterizations of the norm--parallelism for operators on various
Banach spaces and elements of an arbitrary Hilbert $C^*$-module were given in
\cite{B.B, B.C.M.W.Z, G, M.S.P, S.2, W, Z, Z.M.1, Z.M.2}.

In particular, for bounded linear operators $T, S$ on a Hilbert space
$(H, [\cdot, \cdot])$, it was proved in \cite[Theorem 3.3]{Z.M.1}
that $T\parallel S$ if and only if there exists a sequence of unit
vectors $\{\xi_{n}\}$ in $H$ such that
\begin{align*}
\lim_{n\rightarrow \infty } \big|[T\xi_{n}, S\xi_{n}]\big| = \|T\|\,\|S\|.
\end{align*}
Further, for compact operators $T, S$ it was obtained in
\cite[Theorem 2.10]{Z} that $T\parallel S$ if and only if
there exists a unit vector $\xi \in H$ such that
\begin{align*}
\big|[T\xi , S\xi ]\big| = \|T\|\,\|S\|.
\end{align*}
In \cite{Z.M.2}, the authors also considered the characterization of norm
parallelism problem for operators when the operator norm is replaced by
the Schatten $p$-norm ($1<p<\infty $). More precisely, it was proved in
\cite[Proposition 2.19]{Z.M.2} that
$T\parallel S$ in the Schatten $p$-norm if and only if
\begin{align*}
\Big|{\mathrm{tr}}(|T|^{p-1}U^{*}S)\Big| = {\|T\|_p}^{p -1}\|S\|_p,
\end{align*}
where $T = U|T|$ is the polar decomposition of $T$.

Some other related topics can be found in \cite{A.R, C.L.W, G.S.P, G, M.Z, Z},
and the references therein.

It is our aim in the next section to give characterizations of the
norm--parallelism in $\mathcal{C}_b(\Omega)$ and $\mathcal{C}(K)$.
More precisely, for $f, g\in \mathcal{C}_b(\Omega)$ we prove that
$f\parallel g$ if and only if there exists a sequence of probability
measures $\mu_n$ concentrated at the set
$\{x\in \Omega: \, |f(x)| \geq \|f\| - \varepsilon\}$ $(\varepsilon >0)$
such that
\begin{align*}
\Big|\lim_{n\rightarrow\infty}\int_\Omega \overline{f(x)}g(x)d\mu_n(x)\Big|
= {\|f\|}^{-1}\|g\|\lim_{n\rightarrow\infty}\int_\Omega |f(x)|^2d\mu_n(x).
\end{align*}
Moreover, for $f, g\in \mathcal{C}(K)$ we show that $f\parallel g$
if and only if there exists a probability measure $\mu$ with support
contained in the norm attaining set $\{x\in K: \, |f(x)| = \|f\|\}$
such that
\begin{align*}
\Big|\int_K \overline{f(x)}g(x)d\mu(x)\Big| = \|f\|\,\|g\|.
\end{align*}
Finally, in the next section, we state a characterization of the norm-parallelism for
uniformly continuous $X$ valued functions on $\mathbb{B}_X$ to the identity function.
Actually, we show that if $X$ is a Banach space and $f\in\mathcal{C}_u(\mathbb{B}_X,X)$,
then $f\parallel Id$ if and only if $\|f\| = v(f)$.

\section{Main results}
We begin with the following results, which will be useful in other contexts as well.
\begin{lemma}\cite[Theorem 4.1]{Z.M.1}\label{T.1.1}
Let $X$ be a normed space. For $x , y\in X$ the following statements are equivalent:
\begin{itemize}
\item[(i)] $x\parallel y$.
\item[(ii)] There exists a norm one linear functional $\varphi$ over $X$ such that
$\varphi(x) = \|x\|$ and $|\varphi(y)| = \|y\|.$
\end{itemize}
\end{lemma}
\begin{lemma}\cite[Theorem 3.1]{K}\label{lemma.3.02}
Let $\Omega$ be a locally compact Hausdorff space and let $f, g\in \mathcal{C}_b(\Omega)$.
Then
$$\lim_{t\rightarrow 0^{+}}\frac{\|f + tg\| - \|f\|}{t}
= \inf_{\varepsilon>0}\sup_{x\in M^\varepsilon_f}\mbox{Re}\big(e^{-i\arg (f(x))}g(x)\big),$$
where $M^\varepsilon_f = \{x\in \Omega: \, |f(x)| \geq \|f\| - \varepsilon\}$.
\end{lemma}
We use some techniques of \cite{K} to prove the following theorem.
Recall that a probability measure is a positive measure of total mass $1$.
\begin{theorem}\label{th.2}
Let $\Omega$ be a locally compact Hausdorff space and let $f, g\in \mathcal{C}_b(\Omega)$.
Then the following statements are equivalent:
\begin{itemize}
\item[(i)] For every $\varepsilon >0$, there exists a sequence of probability measures $\mu_n$
concentrated at $M^\varepsilon_f$ such that
$$\|f\|\Big|\lim_{n\rightarrow\infty}\int_\Omega \overline{f(x)}g(x)d\mu_n(x)\Big|
= \|g\|\lim_{n\rightarrow\infty}\int_\Omega |f(x)|^2d\mu_n(x).$$
\item[(ii)] $f\parallel g$.
\end{itemize}
\end{theorem}
\begin{proof}
(i)$\Rightarrow$(ii) Suppose (i) holds. Let $\varepsilon >0$.
So, there exist a sequence of probability measures $\mu_n$
concentrated at $M^\varepsilon_f$ and $\lambda\in\mathbb{T}$ such that
$$\|g\|\lim_{n\rightarrow\infty}\int_\Omega |f(x)|^2d\mu_n(x)
= \lambda \|f\|\lim_{n\rightarrow\infty}\int_\Omega \overline{f(x)}g(x)d\mu_n(x),$$
and hence
\begin{align}\label{e.3.10}
\lim_{n\rightarrow\infty}\int_\Omega \overline{f(x)}\Big(\|g\|f(x) - \lambda \|f\|g(x)\Big)d\mu_n(x) = 0.
\end{align}
Let us now define a linear functional $\varphi: \mbox{span}\big\{f, (\|g\|f - \lambda \|f\|g)\big\} \longrightarrow \mathbb{C}$
by setting
\begin{align*}
\varphi\Big(\alpha f + \beta\big(\|g\|f - \lambda \|f\|g\big)\Big)
= \alpha \|f\| \qquad (\alpha, \beta \in \mathbb{C}).
\end{align*}
We have
\begin{align*}
\Big\|\alpha f + &\beta\big(\|g\|f - \lambda \|f\|g\big)\Big\|^2
\\ &= \sup_{x\in \Omega}\Big|\alpha f(x) + \beta\big(\|g\|f(x) - \lambda \|f\|g(x)\big)\Big|^2
\\& = \sup_{x\in \Omega}\Big(|\alpha|^2|f(x)|^2 +
2\mbox{Re}\big[\overline{\alpha}\beta\overline{f(x)}\big(\|g\|f(x) - \lambda \|f\|g(x)\big)\big]
\\&\qquad\qquad \qquad\qquad\qquad\qquad \qquad\qquad+ |\beta|^2\big|\|g\|f(x) - \lambda \|f\|g(x)\big|^2\Big)
\\ & \geq \Big|\int_\Omega\Big(|\alpha|^2|f(x)|^2 +
2\mbox{Re}\big[\overline{\alpha}\beta\overline{f(x)}\big(\|g\|f(x) - \lambda \|f\|g(x)\big)\big]
\\&\qquad\qquad \qquad\qquad\qquad\qquad \qquad\qquad
+ |\beta|^2\big|\|g\|f(x) - \lambda \|f\|g(x)\big|^2\Big)d\mu_n(x)\Big|
\\ & \geq |\alpha|^2\int_\Omega|f(x)|^2 d\mu_n(x)+
2\mbox{Re}\big[\overline{\alpha}\beta\int_\Omega\overline{f(x)}\big(\|g\|f(x) - \lambda \|f\|g(x)\big)d\mu_n(x)\big]
\\& \qquad \qquad \qquad \qquad \qquad \qquad \big(\mbox{since $\int_\Omega |\beta|^2\big|\|g\|f(x) - \lambda \|f\|g(x)\big|^2d\mu_n(x)\geq 0$}\big)
\\ &\geq |\alpha|^2(\|f\| - \varepsilon)^2 +
2\mbox{Re}\big[\overline{\alpha}\beta\int_\Omega\overline{f(x)}\big(\|g\|f(x) - \lambda \|f\|g(x)\big)d\mu_n(x)\big].
\\& \qquad \qquad \qquad \quad \quad \quad\big(\mbox{since $\mu_n$ is a probability measure concentrated at $M^\varepsilon_f$}\big)
\end{align*}
Thus
\begin{align}\label{e.3.11}
\Big\|\alpha f + &\beta\big(\|g\|f - \lambda \|f\|g\big)\Big\|^2\nonumber
\\ &\geq |\alpha|^2(\|f\| - \varepsilon)^2
+ 2\mbox{Re}\big[\overline{\alpha}\beta\int_\Omega\overline{f(x)}\big(\|g\|f(x) - \lambda \|f\|g(x)\big)d\mu_n(x)\big].
\end{align}
Letting $\varepsilon \rightarrow 0^+$ and $n \longrightarrow \infty$ in (\ref{e.3.11}), then by (\ref{e.3.10}) we obtain
\begin{align*}
\Big\|\alpha f + \beta\big(\|g\|f - \lambda \|f\|g\big)\Big\|\geq |\alpha|\|f\|.
\end{align*}
Hence
$\Big|\varphi \Big(\alpha f + \beta\big(\|g\|f - \lambda \|f\|g\big)\Big)\Big|\leq \Big\|\alpha f
+ \beta\big(\|g\|f - \lambda \|f\|g\big)\Big\|$
and $\varphi(f) = \|f\|$. Thus $\|\varphi\| = 1$. Therefore, the Hahn-Banach theorem extends $\varphi$ to a linear functional
$\widetilde{\varphi}$ on $\mathcal{C}_b(\Omega)$, with $\|\widetilde{\varphi}\| = 1$.
Since $\widetilde{\varphi}(f) = \|f\|$ and $\widetilde{\varphi}\big(\|g\|f - \lambda \|f\|g\big) = 0$,
we get $\widetilde{\varphi}(\lambda g) = \|g\|$, hence $|\widetilde{\varphi}(g)| = \|g\|$.
Now, by Lemma \ref{T.1.1}, we conclude that $f\parallel g$.

(ii)$\Rightarrow$(i) Let $f\parallel g$. By Lemma \ref{T.1.1},
there exists a norm one linear functional $\varphi$ over $\mathcal{C}_b(\Omega)$
such that $\varphi(f) = \|f\|$ and $|\varphi(g)| = \|g\|$.
So, there exists $\lambda\in\mathbb{T}$ such that $\varphi(\lambda g) = \|g\|$.
From $\|\varphi\| = 1, \varphi(f) = \|f\|$ and $\varphi(\lambda g) = \|g\|$ it follows that
\begin{align*}
\frac{\Big\|f + re^{i\theta}\big(\|g\|f - \lambda \|f\|g\big)\Big\| - \|f\|}{r}
&\geq \frac{\Big|\varphi\Big(f + re^{i\theta}\big(\|g\|f - \lambda \|f\|g\big)\Big)\Big| - \|f\|}{r}
\\& = \frac{\Big|\varphi(f) + re^{i\theta}\|g\|\varphi(f) - re^{i\theta}\|f\|\varphi(\lambda g)\Big| - \|f\|}{r}
\\& =  \frac{\Big|\|f\| + re^{i\theta}\|g\|\|f\| - re^{i\theta}\|f\|\|g\|\Big| - \|f\|}{r} = 0
\end{align*}
for all $r>0$ and all $\theta\in [0, 2\pi)$. Hence, by Lemma \ref{lemma.3.02}, we get
$$\inf_{\varepsilon>0}\sup_{x\in M^\varepsilon_f}
\mbox{Re}\Big(e^{i\theta} e^{-i\arg (f(x))}\big(\|g\|f(x) - \lambda \|f\|g(x)\big)\Big)\geq 0$$
for all $\theta\in [0, 2\pi)$.
Thus for all $\varepsilon>0$ the set
$$K_\varepsilon := \{e^{-i\arg (f(x))}\big(\|g\|f(x) - \lambda \|f\|g(x)\big):\,\,x\in M^\varepsilon_f\}$$
contains at least one element with nonnegative real part under all rotations around the origin.
Hence the values of the function $\|g\|f - \lambda \|f\|g$ on $M^\varepsilon_f$
are not contained in an open half plane with boundary that contains the origin.
So, for all $\varepsilon>0$ the closed convex hull
of the set $K_\varepsilon$ contains the origin. The convex hull of $K_\varepsilon$
consists of points of the form
$\int_\Omega e^{-i\arg (f(x))}\big(\|g\|f(x) - \lambda \|f\|g(x)\big)d\mu_\varepsilon(x)$,
where $\mu_\varepsilon$ is a probability measure supported on a finite subset of
$M^\varepsilon_f$ (see \cite{R}, chap. 3). Let $n_0\in \mathbb{N}$ and $\frac{1}{n_0}< \varepsilon$.
Thus for every $n\geq n_0$, there is a
probability measure $\mu_n$ concentrated at $M^\varepsilon_f$ such that
\begin{align}\label{e.3.12}
\Big|\int_\Omega e^{-i\arg (f(x))}\big(\|g\|f(x) - \lambda \|f\|g(x)\big)d\mu_n(x)\Big|< \frac{1}{n}.
\end{align}
We have
\begin{align}\label{e.3.121000}
\Big|\int_\Omega &\overline{f(x)}\Big(\|g\|f(x) - \lambda \|f\|g(x)\Big)d\mu_n(x)\Big|\nonumber
\\& = \Big|\int_\Omega \Big(\overline{f(x)} - \|f\|e^{-i\arg (f(x))}\Big)
\Big(\|g\|f(x) - \lambda \|f\|g(x)\Big)d\mu_n(x)\nonumber
\\&\qquad \qquad \qquad \qquad + \int_\Omega \|f\|e^{-i\arg (f(x))}\Big(\|g\|f(x) -
\lambda \|f\|g(x)\Big)d\mu_n(x)\Big|\nonumber
\\& \leq \Big|\int_\Omega \Big(\overline{f(x)} - \|f\|e^{-i\arg (f(x))}\Big)\Big(\|g\|f(x) -
\lambda \|f\|g(x)\Big)d\mu_n(x)\Big|
\\&\qquad \qquad \qquad \qquad \qquad \qquad
+ \|f\|\Big|\int_\Omega e^{-i\arg (f(x))}\big(\|g\|f(x) - \lambda \|f\|g(x)\big)d\mu_n(x)\Big|.\nonumber
\end{align}
Furthermore, since $\|f\| - \frac{1}{n} \leq |f(x)| \leq \|f\|$
for all $x\in M^\varepsilon_f$ we have
\begin{align}\label{e.3.121001}
\Big|\int_\Omega \Big(\overline{f(x)} - \|f\|e^{-i\arg (f(x))}\Big)&\Big(\|g\|f(x) -
\lambda \|f\|g(x)\Big)d\mu_n(x)\Big| \nonumber
\\&\leq \frac{1}{n} \int_\Omega \Big|\|g\|f(x) - \lambda \|f\|g(x)\Big|d\mu_n(x) \nonumber
\\& \leq \frac{1}{n} \Big(\|g\|\int_\Omega |f(x)|d\mu_n(x) + \|f\|\int_\Omega |g(x)|d\mu_n(x) \Big)\nonumber
\\& \leq \frac{2}{n} \|f\|\|g\|.
\end{align}
By (\ref{e.3.12}), (\ref{e.3.121000}) and (\ref{e.3.121001}) we get
\begin{align}\label{e.3.121}
\Big|\int_\Omega &\overline{f(x)}\Big(\|g\|f(x) - \lambda \|f\|g(x)\Big)d\mu_n(x)\Big|
\leq \frac{2}{n} \|f\|\|g\| + \|f\|\frac{1}{n}.
\end{align}
Taking $\displaystyle{\lim_{n\rightarrow\infty}}$ in (\ref{e.3.121}), we obtain
$$\lim_{n\rightarrow\infty}\int_\Omega \overline{f(x)}\Big(\|g\|f(x)
- \lambda \|f\|g(x)\Big)d\mu_n(x) = 0,$$
and hence
$$\|f\|\Big|\lim_{n\rightarrow\infty}\int_\Omega \overline{f(x)}g(x)d\mu_n(x)\Big|
= \|g\|\lim_{n\rightarrow\infty}\int_\Omega |f(x)|^2d\mu_n(x).$$
\end{proof}
Next, we present a characterization of the norm--parallelism
for continuous functions on a compact Hausdorff space $K$.
We will need the following lemma.
\begin{lemma}\cite[Theorem 3.1]{K}\label{lemma.3.029}
Let $K$ be a compact Hausdorff space and let $f, g\in \mathcal{C}(K)$. Then
$$\lim_{t\rightarrow 0^{+}}\frac{\|f + tg\| - \|f\|}{t}
= \max_{x\in M_f}\mbox{Re}\big(e^{-i\arg (f(x))}g(x)\big),$$
where $M_f = \{x\in K: \, |f(x)| = \|f\|\}$.
\end{lemma}
\begin{theorem}\label{th.3}
Let $K$ be a compact Hausdorff space and let $f, g\in \mathcal{C}(K)$.
Then the following statements are equivalent:
\begin{itemize}
\item[(i)] There exists a probability measure $\mu$
with support contained in the norm attaining set $M_f$ such that
$$\Big|\int_K \overline{f(x)}g(x)d\mu(x)\Big| = \|f\|\,\|g\|.$$
\item[(ii)] $f\parallel g$.
\end{itemize}
\end{theorem}
\begin{proof}
We proceed as in the proof of Theorem \ref{th.2}.

(i)$\Rightarrow$(ii) Suppose (i) holds. So, there exists $\lambda\in\mathbb{T}$ such that
\begin{align*}
\int_K \overline{f(x)}\Big(\|g\|f(x) - \lambda \|f\|g(x)\Big)d\mu(x) = 0.
\end{align*}
Thus
\begin{align*}
\Big\|\alpha f + &\beta\big(\|g\|f - \lambda \|f\|g\big)\Big\|^2
\\ &\geq |\alpha|^2\|f\|^2 + |\beta|^2\int_{M_f}\Big|\|g\|f(x)
- \lambda \|f\|g(x)\Big|^2d\mu(x)\geq |\alpha|^2\|f\|^2,
\end{align*}
for any $\alpha, \beta \in \mathbb{C}$.
Let $\varphi: \mbox{span}\big\{f, (\|g\|f - \lambda \|f\|g)\big\} \longrightarrow \mathbb{C}$
be the linear functional defined as
\begin{align*}
\varphi\Big(\alpha f + \beta\big(\|g\|f - \lambda \|f\|g\big)\Big)
= \alpha \|f\| \qquad (\alpha, \beta \in \mathbb{C}).
\end{align*}
Hence $\varphi(f) = \|f\|, \varphi(\lambda g) = \|g\|$ and $\|\varphi\| = 1$.
By the Hahn-Banach theorem, $\varphi$
extends to a linear functional $\widetilde{\varphi}$ on $\mathcal{C}(K)$,
of the same norm. Since $\widetilde{\varphi}(f) = \|f\|$ and $|\widetilde{\varphi}(g)| = \|g\|$,
Lemma \ref{T.1.1} yields $f\parallel g$.

(ii)$\Rightarrow$(i) Let $f\parallel g$. By Lemma \ref{T.1.1},
there exist $\lambda\in\mathbb{T}$ and a norm one linear functional $\varphi$ over
$\mathcal{C}(K)$ such that $\varphi(f) = \|f\|$ and $\varphi(\lambda g) = \|g\|$.
Hence by Lemma \ref{lemma.3.029}, we get
$$\max_{x\in M_f}\mbox{Re}\Big(e^{-i\arg (f(x))}\big(\|g\|f(x) - \lambda \|f\|g(x)\big)\Big)\geq 0.$$
So, the convex hull of the set $\{\overline{f(x)}\big(\|g\|f(x)
- \lambda \|f\|g(x)\big): \,\,x\in M_f\}$ consists of points of the form
$\int_K \overline{f(x)}\big(\|g\|f(x) - \lambda \|f\|g(x)\big)d\mu(x)$,
where $\mu$ is a probability measure supported on a finite subset of $M_f$.
Then there is a sequence $\mu_n$ of probability measures such that
\begin{align*}
\lim_{n\rightarrow\infty}\int_K \overline{f(x)}\big(\|g\|f(x)
- \lambda \|f\|g(x)\big)d\mu_n(x) = 0.
\end{align*}
By the Banach-Alaoglu compactness theorem in dual space, there is a
probability measure $\mu$ such that $\displaystyle{\lim_{i\rightarrow\infty}\mu_{n_i}} = \mu.$
Thus the support of $\mu$ is contained in $M_f$ and we obtain
$$\Big|\int_K \overline{f(x)}g(x)d\mu(x)\Big| = \|f\|\,\|g\|.$$
\end{proof}
As a consequence of Theorem \ref{th.3} we have the following result.
\begin{corollary}\label{cr.4}
Let $K$ be a compact Hausdorff space and let $f, g\in \mathcal{C}(K)$.
If $M_f = \{x_0\}$, then the following statements are equivalent:
\begin{itemize}
\item[(i)] $f\parallel g$.
\item[(ii)] $\{x_0\} \subseteq M_g$.
\end{itemize}
\end{corollary}
We closed this paper with the following equivalence theorem.
More precisely, we state a characterization of the norm--parallelism for
uniformly continuous $X$ valued functions on $\mathbb{B}_X$ to the identity function.
Note that since $\mathbb{B}_X$ is convex and bounded,
then every function in $\mathcal{C}_u(\mathbb{B}_X,X)$ is also bounded.
Before stating our result, let us quote a result from \cite{C.K}.
\begin{lemma}\cite[Corollary 2.4]{C.K}\label{lemma.5.02}
Let $X$ be a Banach space. Let $0< \theta < 2$ and suppose
$y\in \mathbb{B}_X$ and $y^*\in\mathbb{B}_{X^*}$ satisfy
$\mbox{Re}y^*(y) > 1 - \theta$. Then, there are $z\in \mathbb{S}_X$
and $z^*\in \mathbb{S}_{X^*}$ such that
$$z^*(z) = 1, \quad \|y - z\|< \sqrt{2\theta} \quad \mbox{and} \quad \|y^* - z^*\| < \sqrt{2\theta}.$$
\end{lemma}
\begin{theorem}\label{th.30}
Let $X$ be a Banach space and let $f\in\mathcal{C}_u(\mathbb{B}_X,X)$.
Then the following statements are equivalent:
\begin{itemize}
\item[(i)] $v(f) = \|f\|$.
\item[(ii)] $f\parallel Id$,
\end{itemize}
where $Id$ stands for the identity function.
\end{theorem}
\begin{proof}
(i)$\Rightarrow$(ii) Let $v(f) = \|f\|$. For every $\varepsilon > 0$,
we may find $x\in \mathbb{S}_X$ and $x^*\in \mathbb{S}_{X^*}$ such that
$x^*(x) = 1$ and $\big|x^*(f(x))\big| > \|f\| -\varepsilon$.
Let $x^*(f(x)) = \overline{\lambda}\big|x^*(f(x))\big|$ with $\lambda\in\mathbb{T}$. We have
\begin{align*}
1 + \|f\| \geq \|Id + \lambda f\| &\geq  \|x + \lambda f(x)\|
\\& \geq \Big|x^*\big(x + \lambda f(x)\big)\Big| = \Big|x^*(x) + \lambda x^*\big(f(x)\big)\Big|
\\& = \Big|1 + \lambda \overline{\lambda}\big|x^*(f(x))\big|\Big| = 1 + \big|x^*(f(x))\big|> 1 + \|f\| -\varepsilon.
\end{align*}
Thus $$1 + \|f\| \geq \|Id + \lambda f\|> 1 + \|f\| -\varepsilon.$$
Letting $\varepsilon\rightarrow0^+$, we obtain $\|Id + \lambda f\| = 1 + \|f\| $,
or equivalently, $f\parallel Id$.

(ii)$\Rightarrow$(i) Let $f\parallel Id$. So, there exists
$\lambda\in\mathbb{T}$ such that $$\|Id + \lambda f\| = 1 + \|f\|.$$
Fix $0< \varepsilon <1$. Since $f\in\mathcal{C}_u(\mathbb{B}_X,X)$,
there exists $0 < \delta < \varepsilon$ such that
\begin{align}\label{e.3.1032}
\|y - z\|<\delta \Longrightarrow \|f(y) - f(z)\|< \varepsilon \qquad (y, z\in \mathbb{B}_X).
\end{align}
Since $1 + \|f\| = \displaystyle{\sup_{y\in \mathbb{B}_X}}\|y + \lambda f(y)\|$,
there exists $y\in \mathbb{B}_X$ such that
\begin{align*}
\|y + \lambda f(y)\|> 1 + \|f\| - \frac{\delta^2}{2}.
\end{align*}
Then we may find $y^* \in  \mathbb{S}_{X^*}$ such that
\begin{align*}
\mbox{Re}y^*\big(y + \mbox{Re}y^*(\lambda f(y)\big)\big)>1 + \|f\| - \frac{\delta^2}{2},
\end{align*}
which yields
\begin{align}\label{e.3.10327}
\mbox{Re}y^*(y)> 1 - \frac{\delta^2}{2}
\end{align}
and
\begin{align}\label{e.3.10328}
\mbox{Re}y^*(\lambda f(y)) > \|f\| - \frac{\delta^2}{2}.
\end{align}
By (\ref{e.3.10327}) and Lemma \ref{lemma.5.02}, there are
$z\in \mathbb{S}_X$ and $z^*\in \mathbb{S}_{X^*}$ such that
\begin{align}\label{e.3.10329}
z^*(z) = 1, \quad \|y - z\|< \delta \quad \mbox{and} \quad \|y^* - z^*\| < \delta.
\end{align}
So, by (\ref{e.3.10329}) and (\ref{e.3.1032}) we get
\begin{align}\label{e.3.103210}
\|f(y) - f(z)\|< \varepsilon.
\end{align}
By (\ref{e.3.10329}) and (\ref{e.3.103210}) it follows that
\begin{align*}
\Big|\mbox{Re}z^*(\lambda f(z)) - \mbox{Re}y^*(\lambda f(y))\Big|&\leq
\Big|\mbox{Re}z^*\big(\lambda f(z) - \lambda f(y)\big)\Big| + \Big|\mbox{Re}(z^* - y^*)(\lambda f(y))\Big|
\\&\leq \|\lambda f(z) - \lambda f(y)\| + \|z^* - y^*\| \leq \varepsilon + \delta,
\end{align*}
whence
\begin{align}\label{e.3.103211}
\Big|\mbox{Re}z^*(\lambda f(z)) - \mbox{Re}y^*(\lambda f(y))\Big|< \varepsilon + \delta.
\end{align}
So, by (\ref{e.3.10328}) and (\ref{e.3.103211}) we obtain
\begin{align*}
\mbox{Re}z^*(\lambda f(z))> \|f\| - \frac{\delta^2}{2} - (\varepsilon + \delta) > \|f\| - 3\varepsilon.
\end{align*}
This implies
\begin{align*}
\|f\| \geq v(f) \geq \mbox{Re}z^*(\lambda f(z))> \|f\| - 3\varepsilon.
\end{align*}
Letting $\varepsilon\rightarrow0^+$, we conclude $v(f) = \|f\|$.
\end{proof}
As an immediate consequence of Theorem \ref{th.30} we have the following result.
\begin{corollary}\label{cr.40}
Let $X$ be a Banach space and let $f\in\mathcal{C}_u(\mathbb{B}_X,X)$.
If $f\parallel Id$, then $\|f\| = \displaystyle{\sup_{x\in \mathbb{S}_X}}\|f(x)\|.$
\end{corollary}
\textbf{Acknowledgement.}
The author would like to thank the referees for their careful
reading of the manuscript and useful comments.
\bibliographystyle{amsplain}

\begin{thebibliography}{99}

\bibitem{A.R} Lj. Aramba\v{s}i\'{c} and R. Raji\'{c},
\textit{A strong version of the Birkhoff--James orthogonality in Hilbert $C^*$-modules},
Ann. Funct. Anal. \textbf{5} (2014), no. 1, 109--120.

\bibitem{B.B} M. Barraa and M. Boumazgour,
\textit{Inner derivations and norm equality},
Proc. Amer. Math. Soc. \textbf{130} (2002), no. 2, 471--476.

\bibitem{B.C.M.W.Z} T. Bottazzi, C. Conde, M.S. Moslehian, P. W\'{o}jcik and A. Zamani,
\textit{Orthogonality and parallelism of operators on various Banach spaces},
J. Aust. Math. Soc. (to appear).

\bibitem{C.K} M. Chica, V. Kadets, M. Mart\'{i}n, S. Moreno--Pulido, and F. Rambla--Barreno,
\textit{Bishop–-Phelps–-Bollob\'{a}s moduli of a Banach space},
J. Math. Anal. Appl. \textbf{412} (2014), no. 2, 697–-719.

\bibitem{C.L.W} J. Chmieli\'{n}ski, R. {\L}ukasik and P. W\'{o}jcik,
\textit{On the stability of the orthogonality equation
and the orthogonality-preserving property with two unknown functions},
Banach J. Math. Anal. \textbf{10} (2016), no. 4, 828--847.

\bibitem{G.S.P} P. Ghosh, D. Sain and K. Paul,
\textit{On symmetry of Birkhoff-–James orthogonality of linear operators},
Adv. Oper. Theory \textbf{2} (2017), no. 4, 428--434.

\bibitem{G} P. Grover,
\textit{Orthogonality of matrices in the Ky Fan $k$-norms},
Linear Multilinear Algebra, \textbf{65} (2017), no. 3, 496--509.

\bibitem{K} D. Kecki\'{c},
\textit{Orthogonality and smooth points in $C(K)$ and $C_b(\Omega)$},
Eurasian Math. J. \textbf{3} (2012), no. 4, 44--52.

\bibitem{M.S.P} A. Mal, D. Sain and K. Paul,
\textit{On some geometric properties of operator spaces},
Banach J. Math. Anal. (to appear).

\bibitem{M.Z} M. S. Moslehian and A. Zamani,
\textit{Mappings preserving approximate orthogonality in Hilbert $C^*$-modules},
Math Scand. \textbf{122} (2018), 257--276.

\bibitem{R} W. Rudin,
\textit{Functional Analysis},
McGraw–Hill, Inc., New York, 1973.

\bibitem{S} A. Seddik,
\textit{Rank one operators and norm of elementary operators},
Linear Algebra Appl. \textbf{424} (2007), 177--183.

\bibitem{S.2} A. Seddik,
\textit{On the injective norm of $\Sigma_{i=1}^n A_i\otimes B_i$ and characterization of normaloid operators},
Oper. Matrices \textbf{2} (2008), no. 1, 67--77.

\bibitem{We} D. Werner,
\textit{The Daugavet equation for operators on function spaces},
J. Funct. Anal. \textbf{143} (1997), 117--128.

\bibitem{W} P. W\'{o}jcik,
\textit{Norm--parallelism in classical $M$-ideals},
Indag. Math. (N.S.) \textbf{28} (2017), no. 2, 287--293.

\bibitem{Z} A. Zamani,
\textit{The operator-valued parallelism},
Linear Algebra Appl. \textbf{505} (2016), 282--295.

\bibitem{Z.M.1} A. Zamani and M.S. Moslehian,
\textit{Exact and approximate operator parallelism},
Canad. Math. Bull. \textbf{58}(1) (2015), 207--224.

\bibitem{Z.M.2} A. Zamani and M.S. Moslehian,
\textit{Norm--parallelism in the geometry of Hilbert $C^*$-modules},
Indag. Math. (N.S.) \textbf{27} (2016), no. 1, 266--281.

\end{thebibliography}

\end{document}